\documentclass[12pt]{amsart}
\usepackage{lmodern,graphicx}
\usepackage{latexsym}
\usepackage{amsfonts}
\usepackage{pgfplots}
\usepackage{tikz}
\usetikzlibrary{arrows,shapes,positioning}
\usetikzlibrary{decorations.markings}
\tikzstyle arrowstyle=[scale=1]
\tikzstyle directed=[postaction={decorate,decoration={markings,
    mark=at position .65 with {\arrow[arrowstyle]{stealth}}}}]
\tikzstyle reverse directed=[postaction={decorate,decoration={markings,
    mark=at position .45 with {\arrowreversed[arrowstyle]{stealth};}}}]
\usetikzlibrary{backgrounds}
\usepackage[german,english]{babel}

\usepackage{amsmath,amsthm,amssymb}

\DeclareMathOperator{\Aut}{Aut}
\DeclareMathOperator{\rank}{rank}
\DeclareMathOperator{\Epi}{Epi}

\newcommand{\ra}{\rightarrow}

\newcommand{\Z}{\ensuremath{\mathbb{Z}}}

\newtheorem{thm}{Theorem}[section]
\newtheorem{lemma}[thm]{Lemma}

\newtheorem{prop}[thm]{Proposition}
\newtheorem{ex}[thm]{Example}
\newtheorem{claim}{Claim}

\def\acts{\curvearrowright}

\begin{document}

\bigskip
\begin{center}
\Large
Nielsen equivalence in Gupta-Sidki groups
\normalsize

\bigskip

Aglaia Myropolska\footnote{The author acknowledges the support of the Swiss National Science Foundation, grant 200021\_144323.}

\end{center}

\bigskip

\begin{center}
Abstract \end{center} 
For a group $G$ generated by $k$ elements, \emph{the Nielsen equivalence classes} are defined as orbits of the action of $\Aut F_k$, the automorphism group of the free group of rank $k$, on the set of generating $k$-tuples of $G$. 

Let $p\geq 3$ be prime and $G_p$ the Gupta-Sidki $p$-group. We prove that there are infinitely many Nielsen equivalence classes on generating pairs of $G_p$.

\bigskip

\section{Introduction}

Let $G$ be a finitely generated group. \emph{The rank $\rank(G)$} of a group $G$ is the minimal number of generators of $G$. Fix $k\geq \rank(G)$ and let $\Epi(F_k, G)$ be the set of epimorphisms $\phi: F_k\ra G$ from the free group $F_k$ of rank $k$ to $G$. 

Consider the natural action of the group $\Aut F_k\times\Aut G$ on $\Epi(F_k, G)$: for $(\tau, \sigma)~\in~\Aut F_k\times\Aut G$ and for $\phi\in \Epi(F_k, G)$ we define $$\phi^{(\tau, \sigma)}=\sigma\cdot \phi\cdot \tau^{-1}.$$ 
The orbits of this action are called \emph{$T_k$-systems (systems of transitivity)}. B.H.~Neumann and H.~Neumann, motivated by the study of presentations of finite groups, introduced $T_k$-systems in \cite{MR0040297}. One of the main conjectures in this area, sometimes attributed to Wiegold, is that for every finite simple group there is only one system of transitivity when $k\geq 3$\footnote{The classification of finite simple groups implies that every finite simple group can be generated by $2$ elements.}. It is also not known whether there is only one orbit when the action of $\Aut F_k$ with $k\geq 3$ is only considered. 

It was proved by Nielsen \cite{MR1511927} that $\Aut F_k$ is generated by the following automorphisms, where $\{x_1, \dots, x_k\}$ is the basis of $F_k$: 
\begin{align*}
R_{ij}^{\pm}(x_1,\dots,x_i,\dots, x_j,\dots, x_k)&=(x_1,\dots,x_{i}x_j^{\pm 1},\dots,x_j,\dots, x_k),\\
 I_j(x_1,\dots,x_j,\dots,x_k)&=(x_1,\dots,x_j^{-1},\dots,x_k),\\
\end{align*}
where $1\leq i,j \leq k$, $i\neq j$. The transformations above are called \emph{elementary Nielsen moves}. 

Observe that there is a one-to-one correspondence between $\Epi(F_k, G)$ and the set $\{(g_1, \dots, g_k)\mid \langle g_1, \dots, g_k\rangle =G\}$ of generating $k$-tuples of $G$. The action of $\Aut F_k$ on the generating $k$-tuple $(g_1, \dots, g_k)$ is done by applying sequences of elementary Nielsen moves to $(g_1, \dots, g_k)$ by precomposition. 
For example, if $G=\Z^k$ then the set of generating $k$-tuples of $\Z^k$ coincides with $GL(k, \Z)$ and the elementary Nielsen moves induce elementary row operations on the matrices. It follows that the action of $\Aut F_k$ on $\Epi(F_k, \Z^k)$ is transitive.

The orbits of the action $\Aut F_k\acts Epi(F_k, G)$ are called  \emph{Nielsen (equivalence) classes} on generating $k$-tuples of $G$. In recent years the Nielsen equivalence classes became of particular interest as they appear as connected components of the Product Replacement Graph, whose set of vertices coincides with the set $\Epi(F_k, G)$ and the edges correspond to elementary Nielsen moves (see \cite{MR2405933,MR2807845,MR1829489} and Section \ref{Nielsen equiv} for more on this topic). 

Before studying further the Nielsen equivalence, we point out its relation to the famous Andrews-Curtis conjecture \cite{MR0173241}.
Elementary Nielsen moves together with the transformations
\begin{align*} AC_{i,w}(x_1,...,x_i,...,x_k)=(x_1,...,w^{-1}x_iw,...,x_k)
\end{align*} where $1\leq i\leq k$ and $w\in F_k $, form the set of \emph{elementary Andrews-Curtis moves}. Elementary Andrews-Curtis moves transform \emph{normally generating sets} (sets which generate $F_k$ as a normal subgroup) into normally generating sets. 

\emph{The Andrews-Curtis conjecture} asserts that, for a free group $F_k$ of rank $k\geq 2$ and a free basis $(x_1,...,x_k)$ of $F_k$, any normally generating $k$-tuple $(y_1,...,y_k)$ of $F_k$ can be transformed into $(x_1,...,x_k)$ by a sequence of elementary Andrews-Curtis moves.

We say that two normally generating $k$-tuples of $F_k$ are \emph{Andrews-Curtis equivalent} if one is obtained from the other by a finite chain of elementary Andrews-Curtis moves. 
The Andrews-Curtis equivalence corresponds to the actions of $\operatorname{Aut}F_k$ and of $(F_k)^k$ on normally generating $k$-tuples of $F_k$. 
More generally, for a finitely generated group $G$ and $k\geq \rank(G)$, the above actions can be defined on the set of normally generating $k$-tuples of $G$ by precomposition. 
The orbits of this action are called \emph{the Andrews-Curtis equivalence classes in $G$}.
The analysis of Andrews-Curtis equivalence for arbitrary finitely generated groups has its own importance to analyse potential counter-examples to the conjecture. A possible way to disprove the conjecture would be to find two normally generating $k$-tuples of $F_k$ such that their images in some finitely generated group are not Andrews-Curtis equivalent. 
The Andrews-Curtis equivalence was studied for various classes of groups, for instance, for finite groups in \cite{MR2022117,MR2195451}, for free solvable and free nilpotent groups in \cite{MR744511}, for the class $\mathfrak{C}$ of finitely generated groups of which every maximal subgroup is normal in \cite{Myropolska}.
The class $\mathfrak{C}$ includes finitely generated nilpotent groups; moreover all Grigorchuk groups and GGS groups, \emph{e.g.} Gupta-Sidki $p$-groups, belong to $\mathfrak{C}$ by \cite{MR1841763,MR2197824}.  In \cite{Maxsub_spinal} the result for GGS groups was generalized: the authors proved that all multi-edge spinal torsion groups acting on the regular $p$-ary rooted tree, with $p$ odd prime, belong to $\mathfrak{C}$. 

Observe that, for a group $G$ in $\mathfrak{C}$, a normally generating set of $G$ is, in fact, a generating set. Therefore, for groups in $\mathfrak{C}$ the partition (of the set of generating $k$-tuples) into Nielsen equivalence classes is a refinement of the partition into Andrews-Curtis classes. We further describe what is known about Nielsen equivalence for some groups in the class $\mathfrak{C}$.

The most well-understood classification of Nielsen equivalence classes is known for finitely generated abelian groups (see \cite{MR0040297, MR1684568,MR2817667}). Namely, if $G$ is a finitely generated abelian group then the action of $\Aut F_k$ on $Epi(F_k, G)$ is transitive when $k\geq \rank(G)+1$. Moreover, if $k=\rank(G)$ then the number of Nielsen equivalence classes is finite and depends on the primary decomposition of $G$ (see Theorem \ref{DiaconisGraham} for details). It also follows from the latter papers that for any finitely generated abelian group there is only one $T_k$-system for any $k\geq \rank(G)$. 

For a finitely generated nilpotent group the action of $\Aut F_k$ is transitive on $Epi(F_k, G)$ when $k\geq \rank(G)+1$ \cite{MR1218122}. 
However when $k=\rank(G)$ the unicity of Nielsen equivalence class generally breaks down. For instance, Dunwoody \cite{MR0153745} showed that to every pair of integers $n>1$ and $N>0$ there exists a finite nilpotent group of rank $n$ and nilpotency class $2$ for which there are at least $N$ $T_n$-systems.

As a generalization of finite nilpotent groups, we consider the family of Gupta-Sidki $p$-groups $\{G_p\}_{p\geq 3}$ where $p$ is odd prime.  
The group $G_p$ is a $p$-group of rank $2$ acting on the rooted $p$-ary tree, every quotient of which is finite and, therefore, nilpotent  (Gupta-Sidki $p$-groups were defined in \cite{MR759409}; the reader can find the definition in Section $2$). It was shown by Pervova \cite{MR2197824} that the groups $G_p$ belong to the class $\mathfrak{C}$. 
This property was the main ingredient in \cite{Myropolska} for proving that there is only one Nielsen equivalence class for  $G_p$ for $k\geq 3=\operatorname{rank}(G_p)+1$. Moreover, for a group belonging to the class $\mathfrak{C}$, it is relevant to analyse Nielsen equivalence classes in the quotient $G/\Phi(G)$, where $\Phi(G)$ is the Frattini subgroup of $G$. 
Namely, if there are two generating $k$-tuples of $G/\Phi(G)$ which \emph{are not} Nielsen equivalent, then their preimages in $G$ are generating $k$-tuples of $G$ which also \emph{are not} Nielsen equivalent
(see the section on the class $\mathfrak{C}$ in \cite{Myropolska} for details).  
Using this argument and also the fact that for $p>3$ there are, by Theorem \ref{DiaconisGraham}, $\frac{p-1}{2}$ Nielsen classes on generating pairs of the quotient $G_p/\Phi(G_p)=(\mathbb{Z}/p\mathbb{Z})^2$, we conclude that there are at least $\frac{p-1}{2}$ Nielsen classes in $G_p$ for $p>3$. 

For $p=3$ the question on the transitivity of the action of $\Aut F_2$ on $\Epi(F_2, G_3)$ is more subtle.
In this paper we show in particular that, although there is only one Nielsen class on generating pairs of $G_3/\Phi(G_3)=(\mathbb{Z}/3\mathbb{Z})^2$, the action of $\operatorname{Aut}F_2$ is not transitive on $\operatorname{Epi}(F_2, G_3)$. A natural question then is  how many orbits this action has. 

There are numerous examples of groups $G$ with infinitely many Nielsen classes when $k=\rank(G)$. These groups can be found among fundamental groups of certain knots (\cite{MR0444782, MR2780752}), one-relator groups (\cite{MR0419617}), relatively free polynilpotent groups (see \cite{MyrNa} and references therein) and many others. 
We show that for the Gupta-Sidki $p$-group, with $p\geq 3$ prime, there are infinitely many Nielsen classes when $k=\rank(G_p)=~2$. To the author's knowledge this is the first known examples of torsion groups that have this property.

\begin{thm} Let $p\geq 3$ be prime and $G_p$ the Gupta-Sidki $p$-group. Then there are infinitely many Nielsen equivalence classes on generating pairs of $G_p$.
\label{main}
\end{thm}


The Gupta-Sidki $p$-group being a subgroup of $\Aut T_p$, the group of automorphisms of the regular $p$-ary rooted tree $T_p$, has natural quotients by $St_G(n)$, the level stabilizer subgroups.
These quotients are finite nilpotent $2$-generated groups with growing nilpotency class. The latter is true since the limit of these quotients in the space of marked $2$-generated groups is the Gupta-Sidki $p$-group itself, which is not finitely presentable \cite{MR904179}.
In the last part of the paper we show that for each $n\geq 1$ the quotient group $G^{(n)}=G_3/St_{G_3}(n+3)$ of the Gupta-Sidki $3$-group has the property that the action $\Aut F_2\acts \Epi(F_2, G^{(n)})$ is not transitive.
Note that there is only one Nielsen equivalence class in $(\Z/3\Z)^2$, the abelianization of each $G^{(n)}$. 
It would be interesting to realize whether the number of Nielsen classes grows with $n$ but this for the moment remains an open question. 
An affirmative answer on this question, in particular, would imply that there were infinitely many Nielsen equivalence classes in $G_3$. Notice, however, that the proof of Theorem \ref{main} does not rely on Proposition \ref{cor}.

\begin{prop} Let $G_3$ be the Gupta-Sidki $3$-group and $St_{G_3}(n)$ the level stabilizer subgroups of $G_3$. Set $G^{(n)}=G_3/St_{G_3}(n+3)$.  Then the action $\Aut F_2\acts \Epi(F_2, G^{(n)})$ is not transitive for any $n\geq 1$.
\label{cor}
\end{prop}

\textbf{Acknowledgement.} The author would like to thank Pierre de la Harpe, Tatiana Nagnibeda and Said Sidki for stimulating discussions on this work, and Laurent Bartholdi for valuable suggestions during the conference ``Growth in Groups'' in Le Louverain.

\section{Preliminaries on groups acting on rooted trees}
Let $X=\{1,2,\dots, d\}$ with $d\geq 2$ be a finite set. 
The vertex set of the rooted tree $T_d$ is the set of finite sequences $\{x_1 x_2\dots x_k : x_i \in X\}$ over $X$; two sequences are connected by an edge when one can be obtained from the other by right-adjunction of a letter in $X$. 
The top node (the root) is the empty sequence $\emptyset$, and the children of $\sigma$ are all the $\sigma s$ for $s\in X$. 
A map $f\colon T_d\ra T_d$ is an \emph{automorphism of the tree} $T_d$ if it is bijective and it preserves the root and adjacency of the vertices. 
An example of an automorphism of $T_d$ is the rooted automorphism $a_{\pi}$, defined as follows: for the permutation $\pi\in Sym(d)$, set $a_{\pi}(s\sigma):=\pi(s) \sigma$. Geometrically it can be viewed as the permutation of $d$ subtrees just below the root $\emptyset$. Denote by $\Aut T_d$ \emph{the group of automorphisms of the tree $T_d$}.


Let $G\leq \Aut T_d$. Denote by $St_G(\sigma)$ the subgroup of $G$ consisting of the automorphisms that fix the sequence $\sigma$, i.e. $$St_G(\sigma)=\{g\in G\mid g(\sigma)=\sigma\}.$$ And denote by $St_G(n)$ the subgroup of $G$ consisting of the automorphisms that fix all sequences of length $n$, i.e. $$St_G(n)=\cap_{\sigma \in X^n} St_G(\sigma).$$

Notice  an obvious inclusion $St_G(n+1)\leq St_G(n)$. 
Moreover, observe that for any $n\geq 0$ the subgroups $St_G(n)$ are normal and of finite index in $G$. 
We therefore have a natural epimorphism between finite groups \begin{equation} G/St_G(n+1)\rightarrow G/St_G(n),\end{equation} for any $n\geq 0$ .



The examples of groups acting on rooted trees include groups of intermediate growth, such as the Grigorchuk group \cite{MR565099} and the Gupta-Sidki $p$-groups \cite{MR696534}. We define the latter family of groups below.

Fix $p\geq 3$ prime and $X=\{1,2,\dots, p\}$. 
Let $\pi=(1,2,\dots, p)$ be the cyclic permutation on $X$. Let $s$ belong to $X$ and $\sigma$ belong to $T_d$. Denote by $x$ the rooted automorphism of $T_p$ defined by  $$x(s\sigma)=\pi(s)\sigma.$$
Denote by $y$ the automorphism of $T_p$ defined by
\begin{equation*}
y(s\sigma)=\begin{cases}
s x(\sigma) &\mbox{if } s = 1 \\ 
s x^{-1}(\sigma) & \mbox{if } s=2\\
s y(\sigma) & \mbox{if } s=p\\
s \sigma & \mbox{otherwise}. 
\end{cases}
\end{equation*}

The \emph{Gupta-Sidki} $p$-group is the group $G_p$ of automorphisms of the tree $T_p$ generated by $x$ and $y$ and we will write $$G_p=\langle x, y\rangle.$$

To shorten the notation for the element $y$ we will simply write $$y=(x, x^{-1}, 1, \dots, 1, y).$$ 
More generally, for any element $g\in G_p$ we can write $g=x^{i}(g_1, \dots, g_p)$ for some $0\leq i\leq p-1$ and $g_1, \dots, g_p\in G_p$.

\medskip
We summarize here \emph{some facts on the Gupta-Sidki $p$-group} which will be used in the following section.
\begin{enumerate}
\item\cite{MR759409} $G_p$ is just-infinite, i.e. every proper quotient of $G_p$ is finite.
\item\cite{MR2197824} All maximal subgroups of $G_p$ are normal.
\item\cite{MR1953179} The abelianization $G_p^{ab}=G_p/[G_p,G_p]$ is isomorphic to $(\mathbb{Z}/p\mathbb{Z})^2$.
\end{enumerate}

\section{Nielsen equivalence in Gupta-Sidki $p$-groups}
For a finitely generated group $G$ and $k\geq \operatorname{rank}(G)$, we define the \emph{Nielsen graph}\footnote[3]{Also called the Extended Product Replacement Graph.}  $N_k(G)$ as follows: \begin{itemize} 
\item[-] the set of vertices consists of generating $k$-tuples, i.e. \begin{equation*}V_{N_k}(G)=\{(g_1,...,g_k)\in G^k\mid \langle g_1,...,g_k \rangle = G\};\end{equation*}
\item[-]  two vertices are connected by an edge if one of them is obtained from the other by an elementary Nielsen move.
\end{itemize}
Observe that the graph $N_k(G)$ is connected if and only if the action of $\operatorname{Aut}F_k$ on $\operatorname{Epi}(F_k,G)$ is transitive.

\label{Nielsen equiv}

\smallskip
Recall, that for a finitely generated group $G$ the \emph{Frattini subgroup} $\Phi(G)=\cap_{M<_{max}G} M$ is defined as the intersection of all maximal subgroups of $G$. Equivalently, the Frattini subgroup of $G$ contains all the \emph{non-generators}, i.e. the elements which can be removed from any generating set. The latter implies the following lemma.

\begin{lemma}[\cite{MR1218122}]
Let $G$ be a group generated by $\{x_1,...,x_k\}$ and let $\varphi_1,...,\varphi_k \in \Phi(G)$. Then $\langle x_1\varphi_1,...,x_k\varphi_k\rangle=G$.
\label{frat}
\end{lemma}

As it was explained in the Introduction, for groups in class $\mathfrak{C}$ (the class of finitely generated groups all maximal subgroups of which are normal), the number of connected components of $N_k(G)$ is bounded below by the number of connected components of $N_k(G/\Phi(G))$. 
Since all maximal subgroups of $G_p$, the Gupta-Sidki $p$-group, are normal, it follows that the quotient $G_p/\Phi(G_p)$ is abelian. Moreover, any generating set of the quotient $G_p/\Phi(G_p)$ can be lifted up to the generating set of $G_p$ \cite{Myropolska}.
Therefore $G_p/\Phi(G_p)$ is a quotient of $(\mathbb{Z}/p\mathbb{Z})^2$ of rank $2$; we deduce that $G_p/\Phi(G_p)\cong (\mathbb{Z}/p\mathbb{Z})^2$. Using the following theorem, we find the number of connected components of the Nielsen graph $N_2((\mathbb{Z}/p\mathbb{Z})^2)$.

\begin{thm}[\cite{MR0040297, MR1684568,MR2817667}] \leavevmode
\\Let $A$ be a finitely generated abelian group with the primary decomposition $A\cong \Z^s\times\mathbb{Z}_{m_1}\times ... \times \mathbb{Z}_{m_r}$ with $r, s\geq 0$ and $m_r|m_{r-1}|...|m_1$. Then $\rank(A)=r+s$ and
\begin{enumerate}
\item $N_k(A)$ is connected if $k>r+s$.
\item if $r=0$, i.e. $A\cong \mathbb{Z}^s$, then $N_s(G)$ is connected;
\item otherwise if $m_r=2$ or $m_r=3$ then $N_{r+s}(A)$ is connected and if $m_r> 3$ then $N_{r+s}(A)$ has $\varphi(m_r)/2$ connected components,
\end{enumerate}
where $\varphi$ is the Euler function (the number of positive integers less than $m_r$ which are coprime with $m_r$).
\label{DiaconisGraham}
\end{thm}

It follows from Theorem \ref{DiaconisGraham} and the arguments before that for $p>3$ the Nielsen graph $N_2(G_p)$ has at least $\frac{p-1}{2}$ connected components. 

To prove Theorem \ref{main} we use an observation by Nielsen (sometimes also attributed to Higman, see lemma \ref{criterion}) as well as an analysis on conjugacy classes in the Gupta-Sidki $p$-group.

\begin{lemma}[Nielsen]
Let $(u,v)$ and $(u', v')$ be two Nielsen equivalent generating pairs of a group $G$. Then the commutator $[u,v]$ is conjugate either to $[u', v']$ or to $[u',v']^{-1}$. 
\label{criterion}
\end{lemma}
The proof of this lemma is a straightforward calculation of commutators of the pairs obtained from $(u,v)$ by the elementary Nielsen moves.

In order to show that two elements are not conjugate in $G_3$, the Gupta-Sidki $3$-group, sometimes we use the finite quotients $G_3/St_{G_3}(n)$ by the $n$-th level stabilizers. Consider a natural epimorphism $$\pi\colon G_3 \rightarrow G_3/St_{G_3}(4).$$ The finite quotient $G_3/St_{G_3}(4)$ can be seen as a subgroup of $Sym(81)$ with $$\pi(x)=(1,28,55)(2,29,56)\dots(27,54,81)$$ and $$\pi(y)=(1,10,19)\dots (9,18,27)(28,46,37)\dots(36,54,45)(55,58,61)\cdot$$ $$(56,59,62)(57,60,63)(64,70,67)(65,71,68)(66,72,69)(73,74,75)(76,78,77).$$ 

Recall that two elements are conjugate in the symmetric group if and only if their cycle types are the same. Therefore if for two elements $g, h\in G_3$ their images $\pi(g)$ and $\pi(h)$ have different cycle types in $Sym(81)$ then, in particular, they are not conjugate in $G_3$. Below all computations in $Sym(81)$ were done using GAP.

\begin{ex}
The elements $yx^{-1}y^{-1}xy$ and $y$ are not conjugate in $G_3$. Indeed, \begin{center}
$\pi(yx^{-1}y^{-1}xy)=(1,22,10,3,24,12,2,23,11) (4,25,13,5,26,14,6,27,15)\cdot (7,19,16)(8,20,17)(9,21,
18)(55,64,79)(56,65,80)(57,66,81)\cdot  (58,67,74,60,69,73,59,68,75)(61,70,78,62,71,
76,63,72,77),$ \end{center} and its cycle type differs from the one of $\pi(y)$.
\label{ex}
\end{ex}



\bigskip
Let $G_3$ be the Gupta-Sidki $3$-group. Set $z_1=[x,y] \in [G_3,G_3]$ and for all $n>1$ set $z_n=(1, 1, z_{n-1})$. The fact that $z_n\in G_3$ follows from~\cite{MR767112}.

\begin{prop}
The elements $[x,yz_{k}]$, $[x,yz_{j}]^{\pm 1}$ and $z_1^{\pm 1}$ are not pairwise conjugate in $G_3$ for any $k,j>2$ such that $k\neq j$.
\label{conj}
\end{prop}
\begin{proof}
We prove the following two claims in order to conclude the proposition:
\begin{claim}
$[x, yz_n]$ is not conjugate to $z_1^{\pm 1}$ for any $n>2$.
\end{claim}

\begin{claim} 
$[x,yz_{k}]$ and $[x,yz_{j}]^{\pm 1}$ are not conjugate for $k, j>2$ and $k\neq j$.
\end{claim}

The claims will be proved by contradiction. We compute that $z_1=[x,y]=(y^{-1}x, x, xy)$ and $[x, yz_n]=(z_{n-1}^{-1}y^{-1}x, x, xyz_{n-1})$.

\emph{Proof of Claim $1$.}
Assume that $[x, yz_n]$ and $z_1^{\pm 1}$ are conjugate, then there exists $g=x^i(g_1, g_2, g_3)\in G_3$ for some integer $i\in [0,2]$, such that $[x, yz_n]=g^{-1}z_1^{\pm 1}g=(g_1^{-1}, g_2^{-1}, g_3^{-1})x^{-i}(y^{-1}x, x, xy)^{\pm 1}x^i(g_1, g_2, g_3)$. 
Observe that $i=0$ because $x$ is not conjugate neither to $(y^{-1}x)^{\pm 1}$ nor to $(xy)^{\pm 1}$. 
Moreover $x$ is not conjugate to $x^{-1}$ in $G_3$ therefore $[x, yz_n]$ can be conjugate only to $z_1$. We will prove that it is not the case. \emph{For this it is enough to show that $xyz_{n-1}$ and $xy$ are not conjugate in $G_3$.} 
We will show it by induction assuming that \begin{equation*}(*) \text{ }
yz_n \text{ and } y \text{ are not conjugate in } G_3 \text{ for any } n\geq 1 
\end{equation*}
and then will show that (*) is indeed the case.
 
Suppose that $xyz_{n-1}$ and $xy$ are conjugate in $G$ then there exists $g=x^i (g_1, g_2, g_3)$ for some integer $i\in [0,2]$ such that $xyz_{n-1}=x(x, x^{-1}, yz_{n-2})=(g_1, g_2, g_3)^{-1}x^{-i}(xy)x^i(g_1, g_2, g_3)$.

\begin{itemize}
\item If $i=0$ then $(x, x^{-1}, yz_{n-2})=(g_3^{-1}, g_1^{-1}, g_2^{-1})(x,x^{-1}, y)(g_1, g_2, g_3)$ and it follows that $g_2yz_{n-2}g_2^{-1}=y$. 

\item If $i=1$ then $(x, x^{-1}, yz_{n-2})=(g_3^{-1}, g_1^{-1}, g_2^{-1})(y, x, x^{-1})(g_1, g_2, g_3)$ and it follows that $xg_2yz_{n-2}g_2^{-1}x^{-1}=y$. 

\item If $i=2$ then $(x, x^{-1}, yz_{n-2})=(g_3^{-1}, g_1^{-1}, g_2^{-1})(x^{-1}, y, x)(g_1, g_2, g_3)$ and it follows that $g_2 yz_{n-2}g_2^{-1}=y$. 
\end{itemize}

By assumption (*) elements $yz_{n-2}$ and $y$ are not conjugate in $G_3$ and we deduce that $xyz_{n-1}$ is not conjugate to $xy$ in $G_3$ modulo assumption (*).

\emph{Proof of the assumption (*)}: $yz_n$ and $y$ are not conjugate in $G_3$ for any $n\geq 1$.
\begin{enumerate}
\item The assumption holds for $n=1$. To see this,  look at the action of $yz_1$ and $y$ on the 4th level of the tree, see Example \ref{ex}. 
\item Suppose (*) is true for $n-1$.
\item Consider $yz_n=(x, x^{-1}, yz_{n-1})$ and suppose it is conjugate to $y=(x,x^{-1}, y)$. Then there exists $g=x^{i}(g_1, g_2, g_3)\in G_3$ with $0\leq i \leq 2$ such that $(g_1, g_2, g_3)^{-1}x^{-i}(x, x^{-1}, y)x^i(g_1, g_2, g_3)=(x, x^{-1}, yz_{n-1})$. Since $x$ is not conjugate neither to $x^{-1}$ nor to $y$ in $G_3$ then $i=0$. Therefore $(g_1^{-1}xg_1, g_2^{-1}x^{-1}g_2, g_3^{-1}yg_3)=(x, x^{-1}, yz_{n-1})$. We obtain the contradiction with the step of induction.
\end{enumerate}

\emph{Proof of Claim $2$.} We will prove Claim $2$ modulo Assumption (*) and (**) below and then in the end prove that both assumptions indeed hold.

\emph{Assumption (*)}: for any $k, j\geq1$ such that $k\neq j$ the elements $yz_k$ and $yz_j$ are not conjugate in $G_3$. 

\emph{Assumption (**)}: for any $n\geq 2$ the element $x$ is not conjugate to $xyz_n$ or $z_n^{-1}y^{-1}x$ in $G_3$. 

We prove Claim $2$ by contradiction. Suppose that there exists $g=x^i (g_1, g_2, g_3)\in G_3$ such that 
\begin{equation*}[x, yz_k]=g^{-1}[x, yz_j]^{\pm 1} g\end{equation*} 
or equivalently 
\begin{equation}
\label{1}
\begin{split}
 (z_{k-1}^{-1}y^{-1}x, x, xyz_{k-1})=\\ 
(g_1^{-1}, g_2^{-1}, g_3^{-1})x^{-i} (z_{j-1}^{-1}y^{-1}x, x, xyz_{j-1})^{\pm 1}&x^i(g_1, g_2, g_3).
\end{split}
\end{equation} 
Observe that $x$ is not conjugate to $x^{-1}$, $z_{j-1}^{-1}y^{-1}x^{-1}$ and $x^{-1}yz_{j-1}$. To see this, look at the quotient $G_3/St_{G_3}(1)\cong \Z/3\Z$ and notice that the images of $x$ and $x^{-1}$ are not conjugate in $\Z/3\Z$. Therefore $[x, yz_k]$ can not be conjugate to $[x, yz_j]^{-1}$. Moreover, it follows from Assumption (**) that $i=0$ in equation (\ref{1}).

\emph{To obtain the contradiction it is sufficient to show that $xyz_{k-1}$ is not conjugate to $xyz_{j-1}$.} Suppose they are conjugate, then there exists $g=x^i (g_1, g_2, g_3) \in G_3$ with $0\leq i\leq 2$ such that 
\begin{equation*}
x(x,x^{-1}yz_{k-2})=(g_1^{-1}, g_2^{-1}, g_3^{-1})x^{-i} x(x,x^{-1}, yz_{j-2})x^i(g_1, g_2, g_3).
\end{equation*}
\begin{itemize}
\item If $i=0$ then $(x, x^{-1}, yz_{k-2})=(g_3^{-1}xg_1, g_1^{-1}x^{-1}g_2, g_2^{-1}yz_{j-2}g_3)$ and it follows that $yz_{k-2}=g_2^{-1}yz_{j-2}g_2$.
\item If $i=1$ then $(x,x^{-1}, yz_{k-2})=(g_3^{-1}yz_{j-2}g_1, g_1^{-1}xg_2, g_2^{-1}x^{-1}g_3)$ and it follows that $yz_{k-2}=g_2^{-1}x^{-1}yz_{j-2}xg_2$.
\item If $i=2$ then $(x, x^{-1}, yz_{k-2})=(g_3^{-1}x^{-1}g_1, g_1^{-1}yz_{j-2}g_2, g_2^{-1}xg_3)$ and it follows that $yz_{k-2}=g_2^{-1}yz_{j-2}g_2$.
\end{itemize}
By Assumption (*), elements $yz_{k-2}$ and $yz_{j-2}$ are not conjugate in $G_3$ and we deduce that $xyz_{k-1}$ and $xyz_{j-1}$ are not conjugate in $G_3$ modulo assumptions (*) and (**).

\emph{Proof of the assumption (*)} 
Without loss of generality suppose that $j>k$. Suppose $yz_k=(x, x^{-1}yz_{k-1})$ and $yz_j=(x, x^{-1}, yz_{j-1})$ are conjugate. Then there exists $g=x^i(g_1, g_2, g_3)\in G_3$ with $0\leq i\leq 2$ such that \begin{equation*}
(x, x^{-1},yz_{k-1})=(g_1, g_2, g_3)^{-1}x^{-i}(x, x^{-1}, yz_{j-1})x^i(g_1, g_2, g_3).
\end{equation*}
Since $x$ is not conjugate to $x^{-1}$ or to $yz_{j-1}$ we conclude that $i=0$ and hence $yz_{k-1}$ and $yz_{j-1}$ are conjugate. Continuing in the same way, we deduce that the elements $yz_1=(xy^{-1}x, 1, yxy)$ and $yz_{j-k+1}=(x, x^{-1}, yz_{j-k})$ are conjugate. We obtain a contradiction since $x$ is not conjugate to $xy^{-1}x$ or to $yxy$ (to see this it is enough to look at the action of these elements on the 4th level of the tree) or to $1$.

\emph{Proof of the assumption (**)}
To see that $x$ is not conjugate to $xyz_2$ or $z_2^{-1}y^{-1}x$, it is enough to look at the action of these elements on the third level of the tree and to see that they have different cycle types, hence they are not conjugate in the quotient $G_3/St_{G_3}(3)$. 
And for $n\geq 3$, the action of $z_n$ on the third level is trivial therefore it is enough to look at the action of $x$, $xy$ and $y^{-1}x$ on the third level to see that they have different cycle types and therefore not conjugate in $G_3/St_{G_3}(3)$. \end{proof}

\bigskip
Let $G_p$ be the Gupta-Sidki $p$-group for $p\geq 5$ prime. Set $z_1=[x,y]$ and for $n>1$ set $z_n=(1,\dots,1, z_{n-1})$. The fact that $z_n\in G_p$ follows from \cite{MR767112}. 

\begin{prop}
\label{conj_p}
For any $k, j>2$ and $k\neq j$ the elements $[x, yz_k]$ and $[x, yz_j]^{\pm 1}$ are not conjugate in $G_p$.
\end{prop}
\begin{proof}
By contradiction, suppose that there exists an element $$g=x^{i}(g_1, \dots, g_p)~\in~G_p$$ with $0\leq i\leq p-1$ such that 
\begin{equation*}
[x, yz_k]=g^{-1}[x, yz_j]^{\pm 1} g
\end{equation*}
or, in other words,

\begin{equation}
\label{3}
\begin{split} &(z_{k-1}^{-1}y^{-1}x, x^{p-2}, x, 1,\dots, 1, yz_{k-1})=\\ 
(g_1^{-1},\dots, g_p^{-1})x^{-i}&(z_{j-1}^{-1}y^{-1}x, x^{p-2}, x, 1,\dots, 1, yz_{j-1})^{\pm 1} x^i(g_1, \dots, g_p)
\end{split}
\end{equation}
Suppose $i\neq 0$. Observe that $x$ is not conjugate to $1$, $x^{-1}$, $x^{-1}yz_{j-1}$, $x^{p-2}$, $x^{2}$, and to $(yz_{j-1})^{\pm 1}$. To see this, look at the quotient $G_p/St_{G_p}(1)\cong \Z/pZ$, and notice that the image of $x$ is not conjugate to the images of the elements above. 
Therefore $x$ must be conjugate to $z_{j-1}^{-1}y^{-1}x$, in other words there exists $h=x^m(h_1,\dots, h_p) \in G_p$ with $0\leq m\leq p-1$ such that 
\begin{equation*}
x=(h_1,\dots, h_p)^{-1}x^{-m}\cdot (a_1, \dots, a_p)x\cdot x^{m}(h_1,\dots, h_p),
\end{equation*}
where $a_1=x^{-1}$, $a_2=x$, $a_p=z_{j-2}^{-1}y^{-1}$ and $a_k=1$ otherwise.

It follows that the following system of equations holds:
\begin{center}
$\begin{cases}
h_p^{-1}a_{\pi^{m+1}(1)}h_1&=1\\
h_1^{-1}a_{\pi^{m+1}(2)}h_2&=1\\
\dots &\\
h_{p-1}^{-1}a_{\pi^{m+1}(p)}h_p&=1,
\end{cases}$
\end{center}
where $\pi^{m+1}$ is the $m$th power of the permutation $(1,2,\dots, p)$ and, for each $1\leq r\leq p$, $\pi^{m+1}(r)$ denotes the image of $r$ under $\pi^{m+1}$.

After solving the system one obtains that $$h_p^{-1} a_{\pi^{m+1}(1)}a_{\pi^{m+1}(2)}\dots a_{\pi^{m+1}(p)}h_p=1,$$ which gives us a contradiction to $i\neq 0$. 

In view of equation $(\ref{3})$ and that $i=0$, in order to obtain a contradiction to the initial assumption that $[x, yz_k]$ is conjugate to $[x, yz_j]^{\pm 1}$, it is enough to show that $yz_{k-1}$ is not conjugate to $yz_{j-1}$. Without loss of generality suppose that $k>j$. 

Suppose by contradiction that $yz_{k-1}$ is conjugate to $yz_{j-1}$, i.e. there exists $h=(h_1, \dots, h_p)x^l\in G_p$ with $0\leq l\leq p-1$ such that 

\begin{equation*}
\begin{split}
(x, x^{-1}, 1, \dots, 1, yz_{k-2})=\\(h_1,\dots, h_p)^{-1}x^{-l} (x, x^{-1}, 1, \dots, 1, yz_{j-2})& x^l (h_1,\dots, h_p).
\end{split}
\end{equation*}

Observe that $x$ is not conjugate to $1$, $x^{-1}$ and $yz_{j-2}$. 
Hence $l=0$ and therefore $yz_{k-2}$ is conjugate to $ yz_{j-2}$. 
We repeat the same arguments $j-2$ times to conclude that $yz_{k-j+1}=(x, x^{-1}, 1, \dots, 1, yz_{k-j})$ and $yz_1=(xy^{-1}x, x^{p-3},x, 1,\dots, 1, y^2)$ are conjugate. 
Observe that $x^{-1}$ is not conjugate to $1$, $xy^{-1}x$, $x^{p-3}$, $x$ and $y^2$. The contradiction then follows and we deduce that $yz_{k-1}$ is not conjugate to $yz_{j-1}$ which concludes the proof.

\end{proof}

We are now able to deduce that there are infinitely many Nielsen equivalence classes on generating pairs of the Gupta-Sidki $p$-group for any $p\geq 3$ prime.
\begin{proof}[Proof of Theorem \ref{main}]
Fix $p\geq 3$ prime. Let $z_1=[x,y] \in [G_p,G_p]$ and for all $n>1$ let $z_n=(1, \dots, 1, z_{n-1})\in G_p$. It follows from Theorem $4.1.1$  \cite{MR767112} that $z_n\in [G_p,G_p]$. Since $[G_p,G_p]=\Phi(G_p)$ and $\langle x,y\rangle =G_p$ then by Lemma \ref{frat} we deduce that $\langle x, yz_n\rangle =G_p$.
We conclude by Lemma \ref{criterion}, Proposition \ref{conj} and Proposition \ref{conj_p} that there are infinitely many orbits of the action $\Aut F_2\acts \Epi(F_2, G_p)$.
\end{proof}

\begin{proof}[Proof of Proposition \ref{cor}]
First, we show that the graph $N_2(G_3/St_{G_3}(4))$ is not connected. Consider two pairs $(u,v)=(x, y)$ and $(u', v')=(x^{-1}y^{-1}xy\cdot x, y)$ in $G_3$.  Since $\langle x,y\rangle=G_3$ and $[G_3,G_3]=\Phi(G_3)$, it follows that $(u', v')$ is also a generating pair of $G_3$ by Lemma \ref{frat}.

Denote the images of $(u,v)$ and $(u', v')$ in the finite quotient $G_3/St_{G_3}(4)$ by $(\overline{u},\overline{v})$ and $(\overline{u'},\overline{v'})$. Clearly the pairs $(\overline{u},\overline{v})$ and $(\overline{u'},\overline{v'})$ are generating. If they are Nielsen equivalent then  by Nielsen criterion (Lemma \ref{criterion}) their commutators $[\overline{u},\overline{v}]$ and $[\overline{u'},\overline{v'}]^{\pm 1}$ must be conjugate in $Sym(81)$ and, in particular, their cycle types must be the same. We will obtain the contradiction with the latter. 

We calculate the commutators respectively :

\footnotesize
$[\overline{u},\overline{v}]=(1,16,19,3,18,21,2,17,20)(4,10,22,5,11,23,6,12,24)(7,13,25)(8,14,26)$ $(9,15,
27)(28,37,46)(29,38,47)(30,39,48)(31,40,49)(32,41,50)(33,42,51)(34,43,52)$ $(35,
44,53)(36,45,54)(55,70,79)(56,71,80)(57,72,81)$ $(58,64,74,59,65,75,60,66,73)$ $(61,
67,78,63,69,77,62,68,76)$, 

$[\overline{u'},\overline{v'}]=(1,10,25,2,11,26,3,12,27)(4,15,21)(5,13,19)(6,14,20) (7,17,23,9,16,$\\$22,8,18,
24)$ $(28,41,53,29,42,54,30,40,52)(31,45,47,32,43,48,33,44,46)(34,37,50$\\$,35,38,51,
36,39,49)$ $(55,70,79)(56,71,80)(57,72,81)(58,64,74,59,65,75,60,66,73)$\\$(61,67,78,
63,69,77,62,68,76)$. 

\normalsize

The cycle types of $[\overline{u},\overline{v}]$ and $[\overline{u'},\overline{v'}]^{\pm 1}$ are different therefore $(\overline{u},\overline{v})$ and $(\overline{u'},\overline{v'})$ are not Nielsen equivalent. 

For any $l\geq 4$, there exists an epimorphism from $G_3/St_{G_3}(l)$ to $G_3/St_{G_3}(4)$. We will show
that 
the Nielsen graph $N_2(G_3/St_{G_3}(l))$ is not connected using Gasch\"utz lemma \cite{MR0083993}. Gasch\"utz lemma asserts that if there exists an epimorphism between finite groups $f\colon G\rightarrow H$ and $m\geq \operatorname{rank}(G)$ then for any generating $m$-tuple $(h_1,\dots, h_m)$ of $H$ there exists a generating $m$-tuple $(g_1,\dots, g_m)$ of $G$ with $f(g_i)=h_i$ for $i=1,\dots, m$. 
Hence the generating pairs $(\overline{u},\overline{v})$ and $(\overline{u'},\overline{v'})$ of $G_3/St_{G_3}(4)$ have preimages,  generating pairs in $G_3/St_{G_3}(l)$, which are not Nielsen equivalent. The proof is completed.



\end{proof}

\bibliography{Bibliography}
\bibliographystyle{alpha}

\end{document}